\documentclass[12pt,reqno]{amsart}

\usepackage{verbatim}

\newtheorem{thm}{Theorem}
\newtheorem{prop}[thm]{Proposition}
\newtheorem{cor}[thm]{Corollary}
\newtheorem{lem}[thm]{Lemma}
\newtheorem{citethm}{Theorem}

\theoremstyle{definition}
\newtheorem{rem}{Remark}
\newtheorem{ex}{Example}

\providecommand{\RR}{\mathbb{R}}

\providecommand{\NN}{\mathbb{N}}

\DeclareMathOperator{\id}{id}

\DeclareMathOperator{\ke}{Ker}
\DeclareMathOperator{\im}{Im}
\DeclareMathOperator{\sper}{Sper}
\DeclareMathOperator{\cof}{Cof}
\DeclareMathOperator{\supp}{supp}
\DeclareMathOperator{\ind}{Ind}

\def\mn{\mathcal{M}_n}
\def\mnr{\mathcal{M}_n(R)}
\def\mor{\mathcal{M}_1(R)}
\def\mtr{\mathcal{M}_2(R)}
\def\sn{\mathcal{S}_n}
\def\snr{\mathcal{S}_n(R)}
\def\skr{\mathcal{S}_k(R)}
\def\skmr{\mathcal{S}_{k-1}(R)}

\def\zn{\mathcal{Z}_n}
\def\znr{\mathcal{Z}_n(R)}
\def\idn{\mathbf{I}_n}
\def\idnm{\mathbf{I}_{n-1}
\def\id{\mathbf{I}}}
\def\M{\mathcal{M}}
\def\a{\mathbf{A}}
\def\b{\mathbf{B}}
\def\bb{\mathbf{b}}
\def\c{\mathbf{C}}
\def\d{\mathbf{D}}
\def\E{\mathbf{E}}
\def\e{\mathbf{e}}
\def\F{\mathbf{F}}
\def\gg{\mathbf{g}}
\def\h{\mathbf{H}}
\def\J{\mathcal{J}}
\def\S{\mathcal{G}}
\def\s{\mathbf{G}}
\def\T{\mathcal{T}}
\def\t{\mathbf{T}}
\def\vv{\mathbf{v}}
\def\m{\mathbf{M}}
\def\K{\mathfrak{K}}
\def\p{\mathbf{P}}

\def\R{\mathbf{R}}
\def\Q{\mathcal{Q}}
\def\X{\mathbf{X}}
\def\U{\mathfrak{U}}
\def\u{\mathbf{U}}
\def\V{\mathbf{V}}
\def\Z{\mathbf{Z}}
\begin{document}

\title[Real algebraic geometry for matrices]{Real algebraic geometry for matrices over commutative rings}

\author{J. Cimpri\v c}



\date{\today}

\address{Jaka Cimpri\v c, University of Ljubljana, Faculty of Math. and Phys.,
Dept. of Math., Jadranska 19, SI-1000 Ljubljana, Slovenija. 
E-mail: cimpric@fmf.uni-lj.si. www page: http://www.fmf.uni-lj.si/ $\!\!\sim$cimpric.}

\begin{abstract}
We define and study preorderings and orderings on rings
of the form $M_n(R)$ where $R$ is a commutative unital ring. We extend the
Artin-Lang theorem and Krivine-Stengle Stellens\"atze (both abstract and geometric)
from $R$ to $M_n(R)$. This problem has been open since the seventies when 
Hilbert's 17th problem was extended from usual to matrix polynomials.
While the orderings of $M_n(R)$ are in one-to-one correspondence with
the orderings of $R$, this is not true for preorderings. Therefore, 
our theory is not Morita equivalent to the classical real algebraic geometry.
\end{abstract}

\maketitle

\thispagestyle{empty}

\section{Introduction}

Real algebraic geometry studies sets of the form
$$K_{\{g_1,\ldots,g_m\}}=\{x \in \RR^d \mid g_1(x) \ge 0,\ldots,g_m(x) \ge 0\},$$
where $d \in \NN$ and $g_1,\ldots,g_m \in \RR[x_1,\ldots,x_d]$, and the corresponding preorderings
$$T_{\{g_1,\ldots,g_m\}}=\{\sum_{\varepsilon \in \{0,1\}^m} c_\varepsilon
g_1^{\varepsilon_1} \cdots g_m^{\varepsilon_m} \mid c_{\varepsilon} \in \sum\RR[x_1,\ldots,x_d]^2\}.$$
Its most fundamental result is due to Krivine \cite{kr} and Stengle \cite{st}:

\begin{citethm}  
\label{theorema}
For every $f, g_1,\ldots,g_m \in \RR[x_1,\ldots,x_d]$ we have that:
\begin{enumerate}
\item $K_{\{g_1,\ldots,g_m\}}=\emptyset$ iff $-1 \in T_{\{g_1,\ldots,g_m\}}$.
\item  $f(x)>0$ for every $x \in K_{\{g_1,\ldots,g_m\}}$ iff there exist $t,t' \in T_{\{g_1,\ldots,g_m\}}$
such that $ft=1+t'$. 
\item $f(x) \ge 0$ for every $x \in K_{\{g_1,\ldots,g_m\}}$ iff there exist $t,t' \in T_{\{g_1,\ldots,g_m\}}$
and $k \in \NN$ such that $ft=f^{2k}+t'$. 
\item  $f(x) = 0$ for every $x \in K_{\{g_1,\ldots,g_m\}}$ iff there exists $k \in \NN$ such that 
$-f^{2k} \in T_{\{g_1,\ldots,g_m\}}$.
\end{enumerate}
\end{citethm}
Assertion (2) is called the Positivstellensatz, assertion (3) the Nichtnegativstellensatz and
assertion (4) the semialgebraic Nullstellensatz.

The proof consists of two steps. Firstly, one defines the real spectrum of a commutative unital
ring and proves an abstract version of the Krivine-Stengle theorem. Secondly, one proves that the
real spectrum of $\RR[x_1,\ldots,x_d]$ contains $\RR^d$ as a dense subset (Artin-Lang Theorem).

Motivated by \cite{ss} and \cite{sch2}, we would like to extend 
this theory to matrix polynomials of fixed size $n \in \NN$, 
i.e. from the ring $\RR[x_1,\ldots,x_d]$ to the ring $\mn(\RR[x_1,\ldots,x_d])$
of all $n \times n$ matrices with entries from $\RR[x_1,\ldots,x_d]$.
We will consider sets of the form
$$K_{\{\s_1,\ldots,\s_m\}}=\{x \in \RR^d \mid \s_1(x) \succeq 0,\ldots,\s_m(x) \succeq 0\}$$
where $\s_1,\ldots,\s_m$ belong to the set $\sn(\RR[x_1,\ldots,x_d])$ of all symmetric
matrices from $\mn(\RR[x_1,\ldots,x_d])$
and ``$\succeq 0$'' means ``is positive semidefinite''.
In  Section \ref{sec2} we will define the corresponding preordering $T_{\{\s_1,\ldots,\s_m\}} \subseteq \sn(\RR[x_1,\ldots,x_d])$ 
as the smallest quadratic module in $\mn(\RR[x_1,\ldots,x_d])$ which contains $\s_1,\ldots,\s_m$ and whose intersection with
the set $\RR[x_1,\ldots,x_d]\cdot \idn$, where $\idn$ is the identity matrix, is closed under multiplication.
In Section \ref{sec6} we will prove the following generalization of Theorem \ref{theorema}.
 
\begin{citethm}  
\label{theoremb}
For every $\F,\s_1,\ldots,\s_m\in \sn(\RR[x_1,\ldots,x_d])$ we have that:
\begin{enumerate}
\item $K_{\{\s_1,\ldots,\s_m\}}=\emptyset$ iff $-\idn \in T_{\{\s_1,\ldots,\s_m\}}$.
\item  $\F(x)$ is positive definite for every $x \in K_{\{\s_1,\ldots,\s_m\}}$ iff there exist $\b,\b' \in T_{\{\s_1,\ldots,\s_m\}}$
such that $\F \b=\b \F =\idn+\b'$. 
\item $\F(x) \succeq 0$ for every $x \in K_{\{\s_1,\ldots,\s_m\}}$ iff there exist $\b,\b' \in T_{\{\s_1,\ldots,\s_m\}}$
and $k \in \NN$ such that $\F\b=\b\F=\F^{2k}+\b'$. 
\item  $\F(x) = 0$ for every $x \in K_{\{\s_1,\ldots,\s_m\}}$ iff there exists $k \in \NN$ such that 
$-\F^{2k} \in T_{\{\s_1,\ldots,\s_m\}}$.
\end{enumerate}
The element $\b$ from assertion (2) can always be chosen from the set $\RR[x_1,\ldots,x_d]\cdot \idn$
while the element $\b$ from assertion (3) cannot.
\end{citethm}

In Section \ref{sec6} we will also prove an abstract version of Theorem \ref{theoremb}
for rings of the form $\mn(R)$ where $R$ is a commutative unital 
ring and $n$ is a positive integer. The points and the topologies of the real spectrum 
of the ring $\mn(R)$ are defined in Sections \ref{sec5} and \ref{sec4} respectively.
We will prove that the real spectrum of $\mn(R)$ is homeomorphic to the real spectrum of $R$.
This result will imply a variant of the Artin-Lang Theorem 
for the ring $\mn(\RR[x_1,\ldots,x_d])$.

The main step in the proof  of Theorem \ref{theoremb} is the construction of polynomials
$h_1,\ldots,h_r \in \RR[x_1,\ldots,x_d]$ such that $K_{\{h_1  \idn,\ldots,h_r  \idn\}}=K_{\{\s_1,\ldots,\s_m\}}$
and $T_{\{h_1  \idn,\ldots,h_r  \idn\}} \subseteq T_{\{\s_1,\ldots,\s_m\}}$; see Proposition \ref{prop:ks}.
From Theorem \ref{theorema} (applied to $\det(\F-\lambda \idn)-(-\lambda)^n$ and $h_1,\ldots,h_r,-\lambda \in \RR[x_1,\ldots,x_d,\lambda]$)
 and the Cayley-Hamilton Theorem we then deduce a slightly stronger version of the theorem with
$T_{\{\s_1,\ldots,\s_m\}}$ replaced by $T_{\{h_1  \idn,\ldots,h_r  \idn\}}$.
The advantage of the  preordering $T_{\{h_1  \idn,\ldots,h_r  \idn\}}$ is that it is always finitely generated as a quadratic module
while the advantages of the preordering $T_{\{\s_1,\ldots,\s_m\}}$ are that it always contains $\s_1,\ldots,\s_m$ and that it is
uniquely determined by $\s_1,\ldots,\s_m$. Preorderings
of the form $T_{\{h_1  \idn,\ldots,h_r  \idn\}}$ were already considered in \cite[Chapter 4]{sch2} but their version of the matrix
Krivine-Stengle Theorem is not related to ours.


The last claim of Theorem \ref{theoremb} (about denominators in $\RR[x_1,\ldots,x_d]\cdot \idn$) does not follow from other claims. 
The first part is proved by induction on $n$ using Schur complements; see \cite{c}. For the second part see Example \ref{noncentral}.
It would be interesting to know if the element $\b$ from assertion (3) can be chosen from the set $\RR[x_1,\ldots,x_d]\cdot \idn$ if $m=0$
(i.e. in the case of no constraints)\footnotemark
\footnotetext{Note added in press. Alja\v z Zalar recently showed that the answer to this question is negative. 
I would like to thank him for allowing me to include his example here. For $\F$ he takes a diagonal 
$2 \times 2$ matrix with entries $p$ and $1$ where $p$ is a psd form such that every form $h$ for which 
$h^2 p$ is sos must have a zero (see \cite{rez}, the last paragraph in Section 2). Then he proceeds as in Example \ref{noncentral}.}. 
Then Theorem \ref{theoremb} would imply Artin's Theorem
for matrix polynomials (see \cite{gr} or  \cite{ps} or \cite[Proposition 10]{sch2}).

\section{Quadratic modules and preorderings}
\label{sec2}

Let $n$ be a positive integer, $R$ a commutative unital ring and $\mnr$
the ring of all $n \times n$ matrices over $R$ with transposition as the involution.
The unit of $\mnr$ is the identity matrix $\idn$ and the center $\znr$ of $\mnr$ 
is equal to $R \cdot \idn$.  We also assume that 
$R$ has the trivial involution, hence it can be identified as a $\ast$-ring 
with both $\znr$ and $\mor$. We will use the following notation:
lower case letters for elements of $R$, 
upper case letters for subsets of $R$,
bold lower case letters for vectors over $R$, 
bold upper case letters for matrices over $R$,
calligraphic letters for sets of matrices  and 
fracture letters for families of sets of matrices.

Let $\snr:=\{\a \in \mnr \mid \a^T=\a\}$ be the set of all symmetric $n \times n$
matrices over $R$. A subset $\M$ of $\snr$ is a \textit{quadratic module}
if $\idn \in \M$, $\M+\M \subseteq \M$ and $\a^T \M \a \subseteq \M$ for every $\a \in \mnr$.
The smallest quadratic module which contains a given subset $\S$ of $\snr$ will
be denoted by $\M^n_\S$. It consists of all finite sums of elements of the form
$\a^T \s \a$ where $\s \in \S \cup \{\idn\}$ and $\a \in \mnr$. In particular, a subset $M$ of $R$
is a quadratic module if $1 \in M$, $M+M \subseteq M$ and $r^2 M \subseteq M$ for every $r \in R$.
The smallest quadratic module in $R$ which contains a given subset $G$ of $R$ will be denoted by $M_G$.

Let $\E_{ij}$ be the coordinate matrices in $\mnr$, let $\e_i$ be the coordinate vectors in $R^n$
and let $p \colon \mnr \to R$ be the mapping defined by 
$$p(\a)=\e_1^T \a \e_1.$$

\begin{lem}
\label{plem}
\
\begin{enumerate}
\item For every quadratic module $\M$ in $\mnr$, we have that 
$\M \cap \znr=p(\M) \cdot \idn.$
\item For every subset $\S$ of $\snr$, we have that $p(\M_\S^n)=M_{\S'}$ 
where $\S'=\{\vv^T \s \vv \mid \s \in \S, \vv \in R^n\}$,
and so $\M_\S^n \cap \znr=M_{\S'} \cdot \idn$.
\end{enumerate}
\end{lem}

\begin{proof}
(1) If $\a \in \M \cap \znr$, then $\a$ is a scalar matrix, hence $\a=p(\a) \idn$, so $\a \in p(\M)\idn$.
Conversely, if $\a \in p(\M)\idn$, then $\a$ is a scalar matrix, hence it belongs to $\znr$.
Moreover, for some $\m \in \M$, $\a=p(\m)\idn=\sum_{i=1}^n \E_{1i}^T \m \E_{1i} \in \M$.

(2) We have that $p(\M_\S^n)=\{ p(\m) \mid \m \in \M_\S \} 
= \{ p(\sum_{i,j}\a_{ij}^T \s_i \a_{ij}) \mid \s_i \in \S, \a_{ij} \in \mnr \}
= \{\sum_{i,j}\e_1^T \a_{ij}^T \s_i \a_{ij} \e_1 \mid \s_i \in \S, \a_{ij} \in \mnr \}
= \{ \sum_{i,j} \vv_{ij}^T \s_i \vv_{ij} \mid \s_i \in \S, \vv_{ij} \in R^n \}=M_{\S'}$ 
where $\S'=\{\vv^T \s \vv \mid \s \in \S, \vv \in R^n\}$. The last part now follows from (1).
\end{proof}

In particular, if we identify the ring $\znr=R \cdot \idn$ with the ring $R$, then the set $\M \cap \znr$,
which is a quadratic module in $\znr$, is identified with the set $p(\M)$, which is a 
quadratic module in $R$.

\begin{ex} Let $R=\RR[x,y]$, $n=2$ and $\S=\{\left[ \begin{array}{cc} x & 1 \\ 1 & y \end{array} \right]\}$. 
We claim that the quadratic module $p(\M^2_\S)$ in $R$ is not finitely generated.
\end{ex}

\begin{proof}
By Lemma \ref{plem}, $p(\M^2_\S)=M_{\S'}$ where $\S'=\{a^2 x+b^2 y+2 a b \mid a,b \in R\}$ is infinite.
Suppose that $M_{\S'}=M_G$ for some finite set $G \subseteq R$. We may assume that 
$G=\{g_0,g_1,\ldots,g_k\}$ with $g_0=1$.
Let $\succ_1$ and $\succ_2$ be the graded monomial orderings induced by $x \succ_1 y$ and $y \succ_2 x$
respectively. Since the elements of $G$ are nonnegative on $K_G=K_{\S'}=\{(x,y) \in \RR^2 \mid xy \ge 1\}$,
their leading coefficients with respect to $\succ_1$ or $\succ_2$ are all nonnegative.
It follows that $\deg(\sum_i t_i g_i)=\max_i \deg (t_i g_i)$ for any $t_i \in \sum R^2$.
Pick any $\alpha,\beta \in \RR$. Since $\alpha^2x+\beta^2y-2\alpha \beta \in \S' \subseteq M_G$,
there exist $t_i \in \sum R^2$ such that $\alpha^2x+\beta^2y-2\alpha \beta=t_0+t_1 g_1+\ldots+t_k g_k$.
The degree formula implies that $\deg(t_i g_i) \le 1$ for all $i$.
Since $t_i g_i \ge 0$ on $K_{\S'}$,
it follows that for every $i=1,\ldots,l$ there exist $\alpha_i,\beta_i \in \RR^+$
and $\gamma_i \in \RR$ such that $t_i g_i=\alpha_i^2 x+\beta_i^2 y+\gamma_i$ with $\gamma_i \ge -2\alpha_i \beta_i$.
By comparing coefficients, we get that $\alpha^2=\sum_{i=1}^l \alpha_i^2$,
$\beta^2=\sum_{i=1}^l \beta_i^2$ and $-2\alpha \beta=t_0+\sum_{i=1}^l \gamma_i\ge \sum_{i=1}^l(-2 \alpha_i \beta_i)$.
It follows that $\sqrt{\sum_{i=1}^l \alpha_i^2} \sqrt{\sum_{i=1}^l \beta_i^2}
\le \sum_{i=1}^l \alpha_i \beta_i$, hence $(\alpha_1,\ldots,\alpha_l)$ and $(\beta_1,\ldots,\beta_l)$
are colinear, $t_0=0$ and $\gamma_i=-2 \alpha_i \beta_i$.  Therefore, $\alpha^2x+\beta^2y-2\alpha \beta$
is a constant multiple of an element from $G$. In particular, $G$ is infinite, a contradiction.
\end{proof}

A subset $\T$ of the set $\snr$ is a \textit{preordering} if $\T$ is a quadratic module in $\mnr$
and the set $p(\T)$ is closed under multiplication. By Lemma \ref{plem} the set $p(\T)$ is closed
for multiplication iff the set $\T \cap \znr$ is closed under multiplication.
The smallest preordering in $\mnr$ containing a given set $\S \subseteq \snr$ will be denoted by $\T_\S^n$. 

In particular, a subset $T$ of $R$ is a preordering if $T+T \subseteq T$, $T \cdot T \subseteq T$
and $r^2 \in T$ for every $r \in  R$. The smallest preordering in $R$ which contains a given subset
$G$ of $R$ will be denoted by $T_G$.

\begin{ex}
\label{sos}
Let $\Sigma_n(R)$ be the set of all finite sums of elements of the form $\a^T \a$ where $\a \in \mnr$.
We have that $\T_\emptyset^n = \M_\emptyset^n = \Sigma_n(R)$. Moreover, for every ideal $I$ of $R$, 
$\T^n_{I \cdot \idn} = \M^n_{I \cdot \idn} = \T^n_{\sn(I)} = \M^n_{\sn(I)} = \Sigma_n(R)+\sn(I)$.
\end{ex}

\begin{proof}
To prove the first part, note that the set $p(\Sigma_n(R))$ consists
of all finite sums of squares of elements from $R$, hence it is closed under multiplication.
To prove the second part, we have to show that $\sn(I) \subseteq \M^n_{I \cdot \idn}$
and that $\Sigma_n(R)+\sn(I)$ is a preordering. Namely, the second claim implies that
$\T^n_{\sn(I)} \subseteq \Sigma_n(R)+\sn(I)$ while the first claim implies that
$\Sigma_n(R)+\sn(I) \subseteq \M^n_{I \cdot \idn}$. Other inclusions are clear.

Pick $a\in I$ and note that for any positive integers $i$ and $j$,
$a \E_{ii}=\E_{ii}^T (a \idn) \E_{ii} \in \M^n_{I \cdot \idn}$
and $a(\E_{ij}+\E_{ji})=(\E_{ii}+\E_{ij})^T (a  \idn) (\E_{ii}+\E_{ij})
+\E_{ii}^T (-a \idn) \E_{ii} + \E_{jj}^T (-a \idn) \E_{jj}\in \M^n_{I \cdot \idn}$.
It follows that $\sn(I) \subseteq \M^n_{I \cdot \idn}$.
Let $\pi \colon R \to R/I$ be the canonical mapping and let the mapping
$\pi_n \colon \mnr \to \mn(R/I)$ be defined by $\pi_n([a_{ij}])=[\pi(a_{ij})]$.
Since $\pi_n$ is a $\ast$-homomorphism, $(\pi_n)^{-1}(\Sigma_n(R/I))=\Sigma_n(R)+\mn(I)$.
By the first part, $\Sigma_n(R/I)$ is a preordering and so $\Sigma_n(R)+\sn(I)$
is also a preordering.
\end{proof}

\begin{lem}
For every subset $\S$ of $\snr$, we have that
$$\T_\S^n=\M^n_{\S \cup (\prod \S' \cdot \idn)}$$ where $\prod \S'$ is the set of all finite products of elements 
from $$\S'=\{\vv^T \s \vv \mid \s \in \S, \vv \in R^n\}.$$
\end{lem}

\begin{proof}
The formula $(\vv^T \s \vv)\idn = \sum_{i=1}^n (\vv \e_i^T)^T \s (\vv \e_i^T)$,
which follows from $\idn=\sum_{i=1}^n \e_i \e_i^T$, implies that
$\S' \cdot \idn$ belongs to $\M_\S^n$, hence $\prod \S' \cdot \idn$ belongs to $\T_\S^n$. Therefore
$\M^n_{\S \cup (\prod \S' \cdot \idn)} \subseteq \T_\S^n$. It remains to show that the set
$p(\M^n_{\S \cup (\prod \S' \cdot \idn)}) \cdot \idn$
is closed under multiplication. Since $p(\M^n_{\S \cup (\prod \S' \cdot \idn)})= M_{\S''}$
where $\S''=\{\vv^T \t \vv \mid \t \in \S \cup (\prod \S' \cdot \idn), \vv \in R^n\} = \{\vv^T \t \vv \mid \t \in \S, \vv \in R^n\} \cup
\{\t \vv^T \vv \mid \t \in \prod \S' \cdot \idn, \vv \in R^n\} =\S' \cup \prod \S' =\prod \S'$ is closed under multiplication,
we have that the quadratic module $\M^n_{\S \cup (\prod \S' \cdot \idn)}$ is a preordering and hence,
by minimality, equal to $\T_\S^n$.
\end{proof}

For $\S=\emptyset$, the preordering $\T_\S^n=\Sigma_n(R)$ from Example \ref{sos} has the property 
$\T_\S^n (\T_\S^n \cap \znr) \subseteq \T_\S^n$. We will
show now that there exist preorderings which do not satisfy this property.

\begin{ex}
\label{ex2}
Let $R=\RR[x,y]$, $n=2$, $\s=\left[ \begin{array}{cc} x & 1 \\ 1 & y \end{array} \right]$ and $\S=\{\s\}$. 
We claim that $\T_\S^n (\T_\S^n \cap \znr) \not\subseteq \T_\S^n$.
\end{ex}

\begin{proof}
Clearly, $x \idn \in \T_\S^n \cap \znr$ and $\s \in \T_\S^n$. We claim that $x \s \not\in \T_\S^n$.
Suppose that this is false. Since $\T_\S^n=\M^n_{\S \cup (\prod \S' \cdot \idn)}$, there exist
elements $\a_i,\b_j,\c_{kl} \in \mtr$ and $t_k \in \prod \S'$ such that
$$x \s = \sum_i \a_i^T  \a_i+\sum_j \b_j^T \s \b_j+\sum_{k,l} t_k \c_{kl}^T \c_{kl}.$$
Every element $t_k$ is a product of the elements of the form 
$\vv^T \s \vv=p^2 x+q^2 y+2 p q$ where $p,q \in R$.
By comparing the degrees, we see that all $\a_i$ are constant or linear, $\b_i$ and $\c_{k,l}$ are constant and each
$t_i$ is a product of one or two factors 
with constant $p,q$.

Firstly, we compare entries at position $(2,2)$. We get that $(\a_i)_{21}=(\a_i)_{22}=0$ and 
$(\sum_{k,l} t_k \c_{kl}^T \c_{kl})_{22}=xy$. Secondly, we compare entries at position $(1,1)$
and we get that $(\sum_{k,l} t_k \c_{kl}^T \c_{kl})_{11}=0$, which implies that
$(\a_i)_{11}$ and $(\a_i)_{12}$ are constant multiples of $x$.
Finally, we compare entries at position $(1,2)$ and obtain a contradiction $x=0$.
\end{proof}

\begin{prop}
Let $N$ be a quadratic module in $R$ and let $\mathfrak N$ be the set of all quadratic modules 
on $\mnr$ whose intersection with $\znr$ is equal to $N \cdot \idn$. Then the smallest element of $\mathfrak N$ is
the set $$N^n := \{\sum_i n_i \a_i^T \a_i \mid n_i \in N, \a_i \in \mnr\}$$
and the largest element of $\mathfrak N$ is the set
$$\ind(N) := \{\a \in \snr \mid \vv^T \a \vv \in N \mbox{ for all } \vv \in R^n\}.$$
\end{prop}

\begin{proof}
Clearly, $N^n$ and $\ind(N)$ are quadratic modules which contain $N \cdot \idn$. It is also clear that
$p(\sum_i n_i \a_i^T \a_i)=\sum_i(n_i \sum_j (\a_i)_{j1}^2)$, hence
$p(N^n) \subseteq N$. By Lemma \ref{plem}, it follows that $N^n \cap \znr=p(N)\idn$.
Similarly, for every $\a \in \ind(N)$, $p(\a)=\e_1^T \a \e_1 \in N$,
hence $p(\ind(N)) \subseteq N$ and so $\ind(N) \cap \znr=N \cdot \idn$.
Therefore $N^n$ and $\ind(N)$ belong to $\mathfrak N$. Now pick any $\M$
from $\mathfrak N$. Since $\M$ contains $N \cdot \idn$, it also contains the
smallest quadratic module in $\mnr$ generated by $N \cdot \idn$, namely $N^n$.
On the other hand, for every $\m \in \M$ and every $\vv \in R^n$,
$(\vv^T \m \vv)\idn = \sum_{i=1}^n (\vv \e_i^T)^T \m (\vv \e_i^T)$
belongs to $\M \cap \znr=N \cdot \idn$, hence $\vv^T \m \vv \in N$. It follows
that $\M \subseteq \ind(N)$.
\end{proof}

If $N$ is a preordering in $R$ (i.e. a quadratic module in $R$ which is closed under multiplication), 
then every element of $\mathfrak N$ (including $N^n$ and $\ind(N)$)
is a preordering in $\mnr$.


\section{Prime quadratic modules}
\label{sec3}

A quadratic module $\M$ in $\mnr$ is \textit{proper} if $-\idn \not\in \M$. 
A proper quadratic module $\M$ in $\mnr$ is \textit{prime} if for every $\a \in \snr$ 
and $r \in R$ such that $\a r^2 \in \M$ we have that either $\a \in \M$ or $r \in p(\M \cap -\M)$.
(Equivalently, for every $\a \in \snr$ and $\Z \in \znr$ such that $\a \Z^2 \in \M$
we have that either $\a \in \M$ or $\Z \in \M \cap -\M$.)
The set $\supp \M:=\M \cap -\M$ is called the \textit{support} of the quadratic module $\M$.

\begin{lem}
\label{idlem}
For every prime quadratic module $\M$ the following are true:
\begin{enumerate}
\item If $\m_1+\m_2 \in \supp \M$ for some $\m_1,\m_2 \in \M$ then $\m_1,\m_2 \in \supp \M$.
\item If $2\a \in \M$ for some $\a \in \snr$ then $\a \in \M$.
\item $p(\supp \M)\cdot \snr \subseteq \supp \M$ and $R \cdot \supp \M \subseteq \supp \M$.
\item If $\a b \in \supp \M$ for some $\a \in \snr$ and $b \in R$, then either $\a \in \supp \M$ or $b  \in p(\supp \M)$.
\item The set  $p(\supp \M)$ is a prime ideal in $R$.
\item An element $\a=[a_{ij}] \in \snr$ belongs to $\supp \M$ iff all $a_{ij}$ belong to $p(\supp \M)$.
In other words, $\supp \M = \sn(p(\supp \M))$.
\item The set $\{\b \in \mnr \mid \b^T \b \in \supp \M\}$ is a two-sided ideal in $\mnr$
and its intersection with $\snr$ is equal to $\supp \M$.
\item Let $\S$ be a subset of $\mnr$ and let $\operatorname{ideal}(\S)$ be the two-sided ideal in $\mnr$ generated by $\S$.
If $\s^T \s \in \supp \M$ for every $\s \in \S$, then $\operatorname{ideal}(\S) \cap \snr \subseteq \supp \M$.
\end{enumerate}
\end{lem}

\begin{proof}
Let $\M$ be a prime quadratic module. 

If $\m_1,\m_2 \in \M$ and $\m_1+\m_2 \in \supp \M$, then $-\m_1=\m_2+(-\m_1-\m_2) \in \M+\M \subseteq \M$.
This proves (1).

If $2\a \in \M$ for some $\a \in \snr$, then also $4\a \in \supp \M$ since $\M$ is closed under addition.
Since $\M$ is prime, $4\a \in \supp \M$ implies that either $\a \in \supp \M$ or $2 \in p(\supp \M)$.
However, $2 \in p(\supp \M)$ implies that $-1=1+(-2) \in p(\M)+p(\M) \subseteq p(\M)$, a contradiction
with the assumption that $\M$ is proper.

To prove the first part of (3), take any $m \in p(\supp \M)$ and $\a \in \snr$ and note
that the elements $(\idn+\a)^T (m\,\idn) (\idn+\a)$
and $(\idn-\a)^T  (m\,\idn)(\idn-\a)$ belong to $\supp \M$.
Therefore, their sum ($=4m\a$) also belongs to $\supp \M$.
By (2), $m\a \in \supp \M$.
To prove the second part of (3), take any $\m \in \supp \M$ and $r \in R$. Since the elements
$(1+r)^2 \m$ and  $(1-r)^2 \m$ belong to $\supp \M$, their difference ($=4r \m$) also belongs to $\supp \M$.
By (2), $r\m \in \supp \M$.

To prove (4), pick $\a \in \snr$ and $b \in R$ such that $\a b \in \supp \M$. By (3), it follows that
$\a b^2 \in \supp \M$. Since $\M$ is prime, it follows that either $\a \in \supp \M$ or $b \in p(\supp \M)$ as claimed.

Clearly, (5) follows from (3) and (4).

If $\a=[a_{ij}]$ belongs to $\supp \M$, then $a_{ii}=\e_i^T \a \e_i \in p(\supp \M)$ for all $i$
and $a_{ii}+a_{jj}+2a_{ij} =(\e_i+\e_j)^T \a (\e_i+\e_j) \in p(\supp \M)$ for all $i$ and $j$.
Hence $a_{ij} \in p(\supp \M)$ for all $i$ and $j$ by (2) and Lemma \ref{plem}.
Conversely, if $a_{ij} \in p(\supp \M)$ for all $i$ and $j$, then $a_{ii}\E_{ii}$
and $a_{ij}(\E_{ij}+\E_{ji}) \in \supp \M$ by (3). Since $\M$ is closed under addition, 
it follows that $\a \in \supp \M$. This proves (6).

Suppose that $\b^T \b \in \supp \M$ for some $\b=[b_{ij}] \in \mnr$. By (6), it follows that
$b_{1i}^2+\ldots+b_{ni}^2 \in p(\supp \M)$ for every $i=1,\ldots,n$. By (1) and (4), it follows that
$b_{ij} \in p(\supp \M)$ for all $i,j=1,\ldots,n$. Hence, $\b \in \mn(p(\supp \M))$. Conversely, if
$\b \in \mn(p(\supp \M))$, then $\b^T \b \in \sn(p(\supp \M))$ and so $\b^T \b \in \supp \M$
by (5). It follows that $\{\b \in \mnr \mid \b^T \b \in \supp \M\} = \mn(p(\supp \M))$ which
implies both claims of (7).

Write $\J=\{\b \in \mnr \mid \b^T \b \in \supp \M\}$. 
If $\s^T \s \in \supp \M$ for every $\s \in \S$, then $\S \subseteq \J$.
By (7), it follows that $\operatorname{ideal}(\S) \subseteq \J$ and so
$\operatorname{ideal}(\S) \cap \snr \subseteq \J \cap \snr=\supp \M$.
\end{proof}

For every subset $\S$ of $\snr$ write $\K_\S^n$ for the set of all prime quadratic modules 
in $\mnr$ which contain $\S$.  Proposition \ref{prop:ks} is the main result of this section.
The idea for its proof comes from \cite{sch2}.

\begin{prop}
\label{prop:ks}
For every subset $\S \subseteq \snr$ there exists a subset $\tilde{\S} \subseteq \M_\S^n \cap \znr$ such that 
$\K_\S^n=\K_{\tilde{\S}}^n$. If $\S$ is finite, then $\tilde{\S}$ can also be chosen finite.
\end{prop}

\begin{proof}
In the following we assume that for every $k=2,\ldots,n$, 
$\skmr$ is embedded in the upper left corner of $\skr$ by adding zeros elsewhere.

It suffices to show that for every $k=n,\ldots,2$ and every element $\a$ of $\skr$, there exists a finite subset
$\S_\a$ of $\skmr$ such that $\S_\a \subseteq \M_\a$ and $\K_\a^n=\K_{\S_\a}^n$. 
It follows that for every subset $\S' \subseteq \skr$
the set $\S'':=\bigcup_{\a \in \S'} \S_\a$ is contained both in $\skmr$ and $\M_{\S'}$,
it satisfies $\K_{\S''}^n=\bigcap_{\a \in \S'} \K_{\S_\a}^n=\bigcap_{\a \in \S'} \K_\a^n=\K_{\S'}^n$
and it is finite if $\S'$ is finite. The result follows by induction.

Pick $\a=[a_{ij}] \in \snr$. For every $i,j=1,\ldots,n$ write
$$\tilde{a}_{ij}:=\left\{ \begin{array}{cc} a_{ii} & \text{ if } j=i \\ a_{ii}+a_{jj}+2a_{ij} & \text{ if } j \ne i \end{array} \right.$$
and 
$$\t_{ij}:=\left\{ \begin{array}{cc} \idn & \text{ if } j=i \\ 
\idn+\E_{ji} & \text{ if } j \ne i \end{array} \right.$$
Let $\p_{1i}$ be the permutation matrix that belongs to the transposition $(1i)$. Note that
$$\a_{ij}:=\p_{1i}^T \t_{ij}^T \a \t_{ij} \p_{1i}=
\left[ \begin{array}{cc} \tilde{a}_{ij} & \bb_{ij} \\ \bb_{ij}^T & \c_{ij} \end{array} \right]$$
for some $\bb_{ij}$ and $\c_{ij}$. Now write $\b_{ij}:=\tilde{a}_{ij} (\tilde{a}_{ij}\c_{ij}-\bb_{ij}^T \bb_{ij})$ and 
$$\X_{\pm,ij} := \left[ \begin{array}{cc} \tilde{a}_{ij} & \pm  \bb_{ij} \\  0 & \tilde{a}_{ij} \idnm \end{array} \right]$$ and 
observe that
\begin{equation}
\label{eq1}
\X_{-,ij}^T \a_{ij} \X_{-,ij} = \left[ \begin{array}{cc} \tilde{a}_{ij}^3 & 0 \\  0 & \b_{ij} \end{array} \right]
\end{equation}
and
\begin{equation}
\label{eq2}
\X_{+,ij}^T \left[ \begin{array}{cc} \tilde{a}_{ij}^3 & 0 \\ 0 & \b_{ij} \end{array} \right]  \X_{+,ij} =\tilde{a}_{ij}^4 \a_{ij}.
\end{equation}
If $\a \in \skr$, then the set $\S_\a := \bigcup_{i, j=1}^n\{\tilde{a}_{ij}^3, \b_{ij}\}$ 
is clearly contained in $\skmr$. By the equality (1),
it is also contained in $\M_\a^n$. 

To prove that $\K_\a^n=\K_{\S_\a}^n$, it suffices to find an element $t \in \sum R^2$ such that
$\a t \in \M_{\S_\a}^n$ and $\K_{\a t}^n \subseteq \K_\a^n$. These two properties imply that
$\K_{\S_\a}^n \subseteq \K_{\a t}^n \subseteq \K_\a^n$ and the property $\S_\a \subseteq \M_\a^n$
implies that $\K_\a^n \subseteq \K_{\S_\a}^n$.
Write
$$t := \sum_{i,j=1}^n \tilde{a}_{ij}^4.$$
and note that, by the equality (2) and the definition of $\a_{ij}$,
\begin{equation}
\label{eq3}
\a t=\sum_{i,j=1}^n \p_{1i}^T (\t_{ij}^{-1})^T \X_{+,ij}^T 
\left[ \begin{array}{cc} \tilde{a}_{ij}^3 & 0 \\ 0 & \b_{ij} \end{array} \right]  
\X_{+,ij} \t_{ij}^{-1} \p_{1i}\in \M_{\S_\a}^n.
\end{equation}

Suppose now that $\a t \in \M$ for some prime quadratic module $\M$. Since $t \in \sum R^2$, it follows that $\a t^2 \in \M$,
hence either $\a \in \M$ or $t \in p(\supp \M)$. If $t \in p(\supp \M)$, then by claim (1) of Lemma \ref{idlem}, 
$\tilde{a}_{ij}^4 \in p(\supp \M)$ for all $i,j$. By claim (5) of Lemma \ref{idlem},
it follows that $\tilde{a}_{ij} \in p(\supp \M)$ for all $i,j$. Therefore, $a_{ij} \in p(\supp \M)$
by claim (2) of Lemma \ref{idlem}. By claim (6) of Lemma \ref{idlem}, it follows that $\a \in \supp \M$.
\end{proof}

Proposition \ref{prime1} will be used in the proof of Proposition \ref{morita}.

\begin{prop}
\label{prime1}
If $N$ is a prime quadratic module in $R$, then $\ind(N)$ is a prime quadratic module in $\mnr$.
Moreover,  $\ind(N)$ is the only prime quadratic module on $\mnr$ 
whose intersection with $\znr$ is equal to $N \cdot \idn$. 
\end{prop}

\begin{proof}
Suppose that $\a t^2 \in \ind(N)$ for some $\a \in \snr$ and $t\in R$.
It follows that $(\vv^T \a \vv)t^2=\vv^T(\a t^2)\vv \in N$ for every $\vv \in R^n$ and $t \in R$.
If $N$ is prime, then either $\vv^T \a \vv \in N$ or $t \in \supp N$.
If $t \not\in \supp N$, it follows that $\a \in \ind(N)$. If $-\idn \in \ind(N)$, then
$-1=\e_1^T (-\idn)\e_1 \in N$, a contradiction. Therefore, $\ind(N)$ is prime.

The uniqueness part follows from Proposition \ref{prop:ks}.
Namely, for every $\a \in \snr$, there exists a subset
$\S_\a$ of $\znr \cap \M_\a$ such that $\K_\a^n=\K_{\S_\a}^n$. 
For every prime quadratic module $\M$ in $\mnr$, it follows
$\a \in \M$ iff $\M \in \K_\a^n$ iff $\M \in \K_{\S_\a}^n$ iff $\S_\a \subseteq \M$.
Hence $\M=\{\a \in \snr \mid \S_\a \subseteq \M \cap \znr\}$, i.e. $\M$ is uniquely determined by $\M \cap \znr$.
\end{proof}

In particular $N \mapsto \ind(N)$ and $\M \mapsto p(\M)$ give a one-to-one correspondence
between prime quadratic module on $R$ and prime quadratic modules on $\mnr$.

\begin{rem}
Proposition \ref{prime1} can also be proved in a more conceptual way by observing that
there is a natural 1-1 correspondence between prime quadratic modules in $R$ with support $J$ and
quadratic modules in the field $QF(R/J)$ and a natural 1-1 correspondence between prime quadratic
modules in $\mnr$ with support $\sn(J)$ and quadratic modules in $\mn(QF(R/J))$. 
Theorem 1 from \cite{c2} gives a natural 1-1 correspondence between quadratic
modules in $QF(R/J)$ and quadratic modules in $\mn(QF(R/J))$.
\end{rem}

We continue with an alternative description of $\ind(N)$.

\begin{prop}
\label{prime2}
For every prime quadratic module $N$ on $R$, we have that
$$\ind(N)=\{\a \in \snr \mid \exists t \in R \setminus \supp N \colon \a t^2 \in N^n\}.$$
\end{prop}

\begin{proof}
Denote the right-hand side by $\widehat{N^n}$.
Clearly, $\idn \in \widehat{N^n}$. If $\a,\b \in \widehat{N^n}$ then there exist
$s,t \in R \setminus \supp N$ such that $\a s^2 \in N^n$ and $\b t^2 \in N^n$. 
Hence $(\a+\b)s^2 t^2 =(\a s^2)t^2+(\b t^2)s^2 \in N^n$ and $st \in R \setminus \supp N$. 
If $\d \in \widehat{N^n}$ and $\c \in \mnr$,
then $\d u^2 \in N^n$ for some $u \in R \setminus \supp N$, hence $(\c^T \d \c)u^2=\c^T(\d u^2)\c \in N^n$.
Therefore, $\widehat{N^n}$ is a quadratic module. 

Clearly, $N \subseteq p(\widehat{N^n})$. To prove the opposite inclusion,
take any $x \cdot \idn \in \widehat{N^n} \cap \znr$ and pick $s \in R \setminus \supp N$ 
such that $xs^2 \cdot \idn \in N^n \cap \znr=N \cdot \idn$.
Since $N$ is prime, it follows that $x \in N$.

To show that $\widehat{N^n}$ is prime, pick any $\a \in \snr$ and 
$r \in R$ such that $\a r^2 \in \widehat{N^n}$. 
Pick $t \in R \setminus \supp N$ such that $\a r^2t^2 \in N^n$.
If $r \in R \setminus \supp N$ then also $rt \in R \setminus \supp N$
and so $\a \in \widehat{N^n}$. If $r \in \supp N$ then also $r \in p(\supp N)$.

By the uniqueness part of Lemma \ref{prime1}, it follows that $\widehat{N^n}=\ind{N}$.
\end{proof}

\section{Prime quadratic modules as a topological space}
\label{sec4}

For every prime quadratic module $\M$ in $\mnr$, we write
\begin{eqnarray*}
\M^+  = \{\a \in \snr & \mid & \vv^T \a \vv \in p(\M) \setminus \supp p(\M) \\
& & \mbox{ for all } \vv \in R^n \setminus (\supp p(\M))^n \}.
\end{eqnarray*}
In particular, for every prime quadratic module $N$ in $\mor=R$, we have that $N^+=N \setminus \supp N$ and
\begin{eqnarray*}
(\ind N)^+ & = & \{\a \in \snr \mid \vv^T \a \vv \in N^+  \mbox{ for all } \vv \in R^n \setminus (\supp N)^n \}.
\end{eqnarray*}
For every subset $\S$ of $\snr$ write $\U_\S^n$ for the set of all prime quadratic modules $\M$ on $\mnr$
such that $\S \subseteq \M^+$.

\begin{lem}
\label{prop:us}
For every prime quadratic module $\M$, we have that:
\begin{enumerate}
\item $p(\M^+)=p(\M)^+$,
\item $\idn \in \M^+$, $\M^+ + \M^+ \subseteq \M^+$ and for every $\a \in \M^+$ and $\b \in \mnr$ such that 
$\det \b \not\in \supp p(\M)$  we have that $\b^T \a \b \in \M^+$.
\item For every subset $\S \subseteq \snr$ there exists a subset $\tilde{\S} \subseteq \M_\S \cap \znr$ 
such that $\U_\S^n=\U_{\tilde{\S}}^n$. If $\S$ is finite, then so is $\tilde{\S}$.
\end{enumerate}
\end{lem}

\begin{proof}
The inclusion $p(\M^+) \subseteq p(\M)^+$ follows directly from the definitions of $p$ and $\M^+$. To prove
the opposite  inclusion, it suffices to show that $r\idn \in \M^+$ for every $r \in p(\M)^+$.
For every $\vv =(v_1,\ldots,v_n) \in R^n$, we have that $\vv^T (r\idn) \vv=r \sum_i v_i^2 \in p(M)$.
If $\vv^T (r\idn) \vv \in \supp p(\M)$ for some $\vv$, then either $r \in \supp p(\M)$ or $\sum_i v_i^2 \in \supp p(\M)$.
The first case contradicts the assumption $r \in p(\M)^+$ while the second one implies that $\vv \in (\supp p(\M))^n$.

The first claim of (2) follows from the definition of $\M^+$. The second claim follows from claim (1) of Lemma \ref{idlem}
and the definition of $\M^+$. To prove the third claim pick $\a \in \M^+$ and $\b \in \mnr$.
If $\b^T \a \b \not\in \M^+$, then there exists $\vv \in R^n$ such that $\vv^T \b^T \a \b \vv \not\in p(\M)^+$.
Since $\a \in \M^+$, it follows that $\b\vv \in (\supp p(\M))^n$. Since $\supp p(\M)$ is an ideal, it follows that
$(\det \b) \vv =(\cof \b)^T \b\vv \in (\supp p(\M))^n$. Since $\supp p(\M)$ is prime, it follows that either
$\det \b \in \supp p(\M)$ or $\vv \in (\supp \M)^n$.

The proof of assertion (3) is similar to the proof of Proposition \ref{prop:ks}. Namely, take
$$\a=\left[ \begin{array}{cc} a_{11} & \bb \\  \bb^T & \c \end{array} \right],
\quad
\X_{\pm} := \left[ \begin{array}{cc} a_{11} & \pm  \bb \\  0 & a_{11} \idnm \end{array} \right],
\quad
\d=\left[ \begin{array}{cc} a_{11}^3 & 0 \\ 0 & \b \end{array} \right]$$ 
where $\b=a_{11} (a_{11} \c - \bb^T \bb)$ and observe that $\X_{-}^T \a  \X_{-} = \d$ and $\X_{+}^T \d \X_{+} =a_{11}^4 \a$.
If $\a \in \M^+$, then $a_{11} \in p(\M)^+$ by the assertion (1). Since $\det \X_\pm =(a_{11})^n \in p(\M)^+$,
we can apply assertion (2) to get $\d \in \M^+$. Conversely, if $\d \in \M^+$, then $a_{11}^4 \a \in \M^+$
which implies that $\a \in \M^+$. Therefore $\U_\a^n=\U_\d^n$. By induction, there exists a diagonal matrix
$\E \in \mnr$ such that $\U_\a^n=\U_\E^n$. Then $\tilde{\S}$ is the set of all diagonal entries of $\E$.
\end{proof}

Write $\mathfrak{P}(\mnr)$ for the set of all prime quadratic modules on $\mnr$.
The \textit{Harrison topology} on $\mathfrak{P}(\mnr)$ is the topology generated 
by the sets $\U_\S^n$ where $\S$ is a finite subspace of $\snr$. 
The \textit{constructible topology} on $\mathfrak{P}(\mnr)$ is the topology generated 
by the sets $\U_\S^n$ and $\K_\S^n$ where $\S$ is a finite subspace of $\snr$.

By Proposition \ref{prop:ks} and assertion (3) of Lemma \ref{prop:us}, for every finite subset
$\S$ of $\snr$ there exist elements $g_1,\ldots,g_k,h_1,\ldots,h_l \in R$ such that
$\U_\S^n = \U_{\{g_1 \cdot \idn,\ldots,g_k \cdot \idn\}}^n = \U_{g_1 \cdot \idn}^n \cap \ldots \cap \U_{g_k \cdot \idn}^n$ and 
$\K_\S^n = \K_{\{h_1 \cdot \idn,\ldots,h_l \cdot \idn\}}^n = \K_{h_1 \cdot \idn}^n \cap \ldots \cap \K_{h_l \cdot \idn}^n$.
Therefore, the Harrison topology is already generated by the sets $\U_{r \cdot \idn}^n$, $r \in R$, 
and the constructible topology is already generated by the sets 
$\U_{r \cdot \idn}^n$ and $\K_{r \cdot \idn}^n$, $r \in R$.

\begin{prop}
\label{morita}
The mappings $\ind \colon \mathfrak{P}(R) \to \mathfrak{P}(\mnr)$, $P \mapsto \ind P$, and
$p \colon \mathfrak{P}(\mnr) \to \mathfrak{P}(R)$, $\Q \mapsto p(\Q)$, are homeomorphisms
w.r.t. the Harrison topology and also w.r.t. the constructible topology.
\end{prop}

\begin{proof}
Since $p$ is the inverse of $\ind$, it suffices to prove the claim about $\ind$.
Clearly, for every prime quadratic module $N$ on $R$ 
and every element $r \in R$ we have that $r \in N^+$ iff $r \idn \in (\ind N)^+$
and also that $r \in N$ iff $r \idn \in \ind N$. It follows that for every $r_1,\ldots,r_k \in R$
we have that $N \in \U_{\{r_1,\ldots,r_k\}}^1$ iff $\ind N \in \U_{\{r_1 \idn,\ldots,r_k \idn\}}^n$
and similarly, $N \in \K_{\{r_1,\ldots,r_k\}}^1$ iff $\ind N \in \K_{\{r_1  \idn,\ldots,r_k \idn\}}^n$.
 Hence $\ind$ is open and continuous in both topologies.
\end{proof}

\section{Orderings}
\label{sec5}

Recall that a subset $P$ of a unital commutative ring $R$ is an \textit{ordering} if
$P+P \subseteq P$, $P \cdot P \subseteq P$, $P \cup -P=R$ and $P \cap -P$ is a prime ideal.
Note that every ordering is  a prime quadratic module and a preordering.
For every homomorphism from $R$ into a real closed field $\kappa$, the set $\phi^{-1}(\kappa^2)$,
where $\kappa^2=\sum \kappa^2$ is the only ordering in $\kappa$, is an ordering in $R$. 
Moreover, every ordering on $R$ is of this form. Namely, for every ordering $P$ in $R$ 
write $\kappa_P$ for the real closure of $QF(R/\supp P)$ with respect to the ordering induced by $P$. 
Note that $P=\phi_P^{-1}(\kappa_P^2)$ where $\phi_P \colon R \to \kappa_P$ is the natural 
homomorphism.

A subset $\Q$ of $\snr$ will be called an \textit{ordering} 
if $\Q$ is a prime quadratic module on $\mnr$ and the set $p(\Q)$ is an ordering on $R$.
The set of all orderings on $\mnr$ will be denoted by $\sper \mnr$ and called the 
\textit{real spectrum} of $\sper \mnr$. The Harrison and the constructible topology
on $\sper \mnr$ are inherited from $\mathfrak{P}(\mnr)$. Proposition \ref{morita} 
implies that $\sper \mnr$ is homeomorphic to $\sper R$ in both topologies. Since 
$\K_{r \cdot \idn}^n \cap \sper \mnr = (\sper \mnr) \setminus \U_{-r \cdot \idn}^n$
for every $r \in R$, it follows that for every finite subset $\S$ of $\snr$
the set $\K_\S^n \cap \sper \mnr$ is closed in the relative Harrison topology.

\begin{lem}
\label{indord}
For every real closed field $\kappa$ and every homomorphism $\phi \colon R \to \kappa$, we have that
$$\ind \phi^{-1}(\kappa^2) =\phi_n^{-1}(\sn(\kappa)^2) \cap \snr$$
where $\phi_n \colon \mnr \to \mn(\kappa)$ is defined by $\phi_n([a_{ij}])=[\phi(a_{ij})]$ and 
where $\sn(\kappa)^2$ is the set of all positive semidefinite matrices in $\sn(\kappa)$.
\end{lem}

\begin{proof}
Let us write $\Q=\phi_n^{-1}(\sn(\kappa)^2) \cap \snr$ and $P=\phi^{-1}(\kappa^2)$.
By Proposition \ref{prime1}, it suffices to show that the set $\Q$
is a prime quadratic module on $\mnr$ such that $p(\Q)=P$.
Clearly, $\idn \in \Q$ and $\Q+\Q \subseteq \Q$. If $\a \in \Q$ and $\b \in \mnr$,
then $\phi_n(\b^T \a \b)=\phi_n(\b)^T \phi_n(\a) \phi_n(\b) \in \phi_n(\b)^T \sn(\kappa)^2 \phi_n(\b) \subseteq \sn(\kappa)^2$,
hence $\b^T \a \b \in \Q$.

To show that $\Q$ is prime, pick $\a \in \snr$ and $t \in R$ such that $\a t^2 \in \Q$.
It follows that $\phi_n(\a)\phi(t)^2 \in \sn(\kappa)^2$. Hence, either $\phi(t)=0$ or
$\phi_n(\a) \in \sn(\kappa)^2$. In the first case, we have that $t \, \idn \in \supp \Q$,
and in the second case that $\a \in \Q$.

Pick $\a \in \Q$. It follows that $\phi(p(\a))=\phi(\e_1^T \a \e_1)=\e_1^T \phi_n(\a) \e_1 \subseteq
\e_1^T \sn(\kappa)^2 \e_1 \subseteq \kappa^2$, hence $p(\a) \in P$. Conversely, take any $b \in P$
and note that $\phi_n(b\, \idn)=\phi(b)\idn \in \kappa^2 \cdot \idn \subseteq \sn(\kappa)^2$.
Hence, $b\,\idn \in \Q$ and so $b \in p(\Q)$.
\end{proof}

In summary, every ordering $\Q$ on $\mnr$ is of the form
$$\Q = \ind(P)=\widehat{P^n}=((\phi_P)_n)^{-1}(\sn(\kappa_P)^2) \cap \snr$$
where $P=p(\Q)$ is an ordering on $R$. 
It follows that
$$\Q^+ = ((\phi_P)_n)^{-1}((\sn(\kappa_P)^2)^+) \cap \snr$$
where $(\sn(\kappa_P)^2)^+$ is the  set of all positive definite matrices over $\kappa_P$.

\begin{cor}
\label{corord}
For every $\a \in \snr$ and every ordering $P$ on $R$, the following are equivalent:
\begin{enumerate}
\item $\a \in \ind(P)$,
\item all principal minors of $\a$ belong to $P$,
\item all coefficients of the polynomial $\det(\a+\lambda \idn)$ belong to $P$.
\end{enumerate}
Moreover, the following are equivalent:
\begin{enumerate}
\item $\a \in \ind(P)^+$,
\item $\a \in \ind(P)$ and $\det \a \in P^+$,
\item all leading principal minors of $\a$ belong to $P^+$.
\end{enumerate}
\end{cor}

Corollary \ref{corord} follows from Lemma \ref{indord} and the usual characterizations
of positive semidefinite and positive definite matrices over real closed fields.

For every real closed field $\kappa$ write $V_\kappa(R)$ for the set of all ring homomorphisms from $R$ to $\kappa$.

\begin{thm} (Artin-Lang Theorem)
\label{artinlang}
If $R$ is a finitely generated $\RR$-algebra, then the mapping
$$i \colon V_\RR(R) \to \sper \mnr, \quad \phi \mapsto \phi_n^{-1}(\sn(\RR)^2) \cap \snr$$
is one-to-one and its image is dense in the constructible topology.
\end{thm}

\begin{proof}
By Lemma \ref{indord}, $i$ is the compositum of the mapping $j \colon V_\RR(R) \to \sper R$, $j(\phi)=\phi^{-1}(\RR^2)$,
and the mapping $\ind \colon \sper R \to \sper \mnr$. 
Since $\ind$ is  homeomorphism with respect to the constructible topologies
on $\sper R$ and $\sper \mnr$, the theorem follows from the usual Artin-Lang theorem
\cite[2.4.3. Theorem]{mm}.
\end{proof}

\section{Stellens\" atze}
\label{sec6}

In this section, we will prove a generalization of the Krivine-Stengle theorem 
whose special case is the result that was announced in the Introduction.

Recall that an \textit{affine $\RR$-algebra} is a finitely generated 
commutative unital $\RR$-algebra. 

\begin{thm}(Positivdefinitheitsstellensatz)
\label{thm1}
For every finite subset $\S$ of $\snr$ and every element $\F \in \snr$, the following are equivalent:
\begin{enumerate}
\item[(1)] For every real closed field $\kappa$ and every $\phi \in V_\kappa(R)$
such that $\phi_n(\s)$ is positive semidefinite for every $\s \in \S$, we have
that $\phi_n(\F)$ is positive definite. 
\item[(1')] For every ordering $\Q \in \sper \mnr$ which contains $\S$, we have that $\F \in \Q^+$.
\item[(2)] There exist $\b,\c \in \T_\S^n$ such that $\F(\idn+\b)=(\idn+\b)\F=\idn+\c$.
\item[(2')] There exist $\b,\c \in \T_\S^n$ such that $\F\b=\b\F=\idn+\c$.
\item[(3)] There exist $t \in p(\T_\S^n)$ and $\V \in \T_\S^n$ such that $\F(1+t)=\idn+\V$.
\item[(3')] There exist $t \in p(\T_\S^n)$ and $\V \in \T_\S^n$ such that $\F t=\idn+\V$.
\end{enumerate}
If $R$ is an affine $\RR$-algebra, then assertions (1)-(3') are equivalent to
\begin{enumerate}
\item[(1'')] For every $\phi \in V_\RR(R)$
such that $\phi_n(\s)$ is positive semidefinite for every $\s \in \S$, we have that $\phi_n(\F)$
is positive definite.
\end{enumerate}
\end{thm}

\begin{proof}
The equivalence of (1) and (1') follows from the comments after Lemma \ref{indord} and the
equivalence of (1') and (1'') follows from Theorem \ref{artinlang}.
Clearly, (3) implies (2) and (3') and either (2) or (3') implies (2').

To prove that (2') implies (1), note that
for every $\phi \in V_\kappa(R)$
such that $\phi_n(\s)$ is positive semidefinite for every $\s \in \S$,
we have that $\phi_n(\b)$ and $\phi_n(\c)$ are  positive semidefinite.
Moreover, the relation $\phi_n(\F)\phi_n(\b)=\phi_n(\b)\phi_n(\F)=\idn+\phi_n(\c)$
implies that $\phi_n(\b)$ is positive definite.
It follows that $\phi_n(\F)=\phi_n(\b)^{-1}(\idn+\phi_n(\c))=
\phi_n(\b)^{-1/2}\phi_n(\idn+\phi_n(\c))\phi_n(\b)^{-1/2}$ is positive definite.

The proof that (1) implies (3) is an obvious generalization of the proof of the main 
theorem from \cite{c}. We summarize it here for the sake of completeness.

Suppose that (1) is true. By Proposition \ref{prop:ks}, we can pick a finite set
$\tilde{\S} \subseteq \M_\S^n \cap \znr$ such that $\K_\S^n=\K_{\tilde{\S}}^n$.
Note that $\T_{\tilde{\S}}^n \subseteq \T_\S^n$. 
Pick any real closed field $\kappa$ and any $\phi \in V_\kappa(R)$ such that
$\phi(g) \ge 0$ for every $g \in \tilde{\S}$. Since $\K_\S^n=\K_{\tilde{\S}}^n$,
it follows that $\phi_n(\s)$ is positive semidefinite for every $\s \in \S$.
By (1), it follows that $\phi_n(\F)$ is positive definite. If we write
$$\F=\left[ \begin{array}{cc} f_{11} & \gg \\ \gg^T & \h \end{array} \right]$$ 
then clearly $\phi(f_{11})>0$ and $\phi_{n-1}(f_{11}\h -\gg^T \gg)$ 
is positive definite. 

By \cite[2.5.2 Abstract Positivstellensatz]{mm}, there exist $t,t' \in T_{\tilde{\S}}$
such that $f_{11} t=1+t'$. Note that $s_1:=t+t'$ and $u_1:=t'+f_{11}^2 t$
belong to $T_{\tilde{\S}}$ and satisfy
$$(1+s_1)f_{11} =1+u_1.$$
By induction, there exist $s \in T_{\tilde{\S}}$ and $\u \in \T_{\tilde{\S}}^{n-1}$ such that 
$$(1+s)(f_{11}\h -\gg^T \gg)=\idnm+\u.$$ 
Write 
\begin{eqnarray*}
\tilde{\gg} & = & (1+s_1)\gg,\\
c & = & \tilde{\gg}\tilde{\gg}^T \in T_{\tilde{\S}},\\
\mathbf{S} & = & c\, \idnm-\tilde{\gg}^T\tilde{\gg} \in \T_{\tilde{\S}}^{n-1},\\
v & = & 1+c, \\
d & = & v(1+u_1)^2-1 \in T_{\tilde{\S}},\\
\d & = & (v(1+u_1)^2 +v^2(2u_1+u_1^2))\idnm+(v+1)\mathbf{S} \in \T_{\tilde{\S}}^{n-1},\\
e & = & (1+s)(1+u_1)^2-1 \in T_{\tilde{\S}},\\ 
\E & = & (1+s_1)^2 (\idnm+\u)-\idnm \in \T_{\tilde{\S}}^{n-1},\\
\R & = & \left[ \begin{array}{cc} 1+u_1 & \tilde{\gg} \\ 0 & (1+u_1)\idnm \end{array} \right],\\
\Z & = & \left[ \begin{array}{cc} v (1+u_1) & (1+v)\tilde{\gg} \\ 0 & 0 \end{array} \right]
\end{eqnarray*}
and note that
$$v(1+v)(1+s)(1+s_1)(1+u_1)^3 \left[ \begin{array}{cc} f_{11} & \gg \\ \gg^T & \h  \end{array} \right] = $$
$$=\idn+\left[ \begin{array}{cc} d & 0 \\ 0 & \d \end{array} \right]
+v(1+v)\R^T \left[ \begin{array}{cc} e & 0 \\ 0 & \E \end{array} \right]\R+\Z^T \Z \in \idn+\T_{\tilde{\S}}^n.$$
Clearly, $v(1+v)(1+s)(1+s_1)(1+u_1)^3 \in 1+T_{\tilde{\S}}\subseteq p(\idn+\T_{\S}^n)$ 
and $\idn+\T_{\tilde{\S}}^n \subseteq \idn+\T_{\S}^n$.
\end{proof}

We continue with a generalization of Krivine-Stengle Nichtnegativstellensatz.

\begin{thm}(Positivsemidefinitheitsstellensatz)
\label{thm2}
For every finite subset $\S$ of $\snr$ and every $\F \in \snr$ the following are equivalent:
\begin{enumerate}
\item[(1)] For every real closed field $\kappa$ and every $\phi \in V_\kappa(R)$
such that $\phi_n(\s)$ is positive semidefinite for every $\s \in \S$, we have
that $\phi_n(\F)$ is positive semidefinite. 
\item[(1')] Every ordering on $\mnr$ which contains $\S$ also contains $\F$.
\item[(2)] There exist $k \in \NN$ and $\b,\c \in \T_\S^n$ such that $\F\b=\b\F=\F^{2k}+\c$.
\end{enumerate}
If $R$ is an affine $\RR$-algebra, then (1), (1') and (2) are equivalent to
\begin{enumerate}
\item[(1'')] For every $\phi \in V_\RR(R)$
such that $\phi_n(\s)$ is positive semidefinite for every $\s \in \S$, $\phi_n(\F)$ is also positive semidefinite.
\end{enumerate}
\end{thm}

\begin{proof}
The equivalence of (1) and (1') follows from the comments after Lemma \ref{indord}. The
equivalence of (1') and (1'') follows from Theorem \ref{artinlang}.

To prove that (2) implies (1) note that the relation $\F \b= \b \F=\F^{2k}+\c$ implies
that $\c$ also commutes with $\F$ and $\b$. Pick a homomorphism $\phi$ from $R$ into a real closed field $\kappa$
such that $\phi_n(\s)$ is positive semidefinite for every $\s \in \S$. Note that
$\phi_n(\b)$ and $\phi_n(\c)$ are also positive semidefinite because they belong to $\T_\S^n$.
Clearly $V_1:=\im \phi_n(\F)$ and $V_2:= \ke \phi_n(\F)$ are subspaces of $\kappa^n$ which satisfy $V_1 \oplus V_2=\kappa^n$.
It is also clear that $V_1$ and $V_2$ are invariant for $\phi_n(\F)$, $\phi_n(\b)$ and $\phi_n(\c)$. 
Write $\F_i=\phi_n(\F)|_{V_i}$ and similarly for $\b$ and $\c$. Clearly, $\F_1$ is invertible and $\F_2=0$.
Since $\phi_n(\b)$ and $\phi_n(\c)$ are positive semidefinite so are $\b_1$ and $\c_1$. It follows that
$\F_1^{2k}+\c_1$ is positive semidefinite and invertible, so that $\F_1 \b_1$ is also positive semidefinite
and invertible. It follows that $\b_1$ is invertible. Since $\F_1=(\F_1^{2k}+\c_1)\b_1^{-1}$ is a product
of two commuting positive semidefinite matrices, it is positive semidefinite. Therefore $\phi_n(\F) = \F_1 \oplus 0$
is positive semidefinite as claimed.

It remains to show that (1) implies (2). By Proposition \ref{prop:ks}, we can pick a finite set 
$\tilde{\S} \subseteq \M_\S^n \cap \znr$ such that $\K_\S^n=\K_{\tilde{\S}}^n$.
We identify $\tilde{\S}$ with $p(\tilde{\S}) \subseteq R \subseteq R'$ where $R'=R[\lambda]/(\det(\F-\lambda \idn))$.
For every homomorphism $\phi \colon R' \to \kappa$ where $\kappa$ is a real closed
field the following is true: if $\phi(g) \ge 0$ for every $g \in \tilde{\S}$,
then $\phi(\bar{\lambda})\ge 0$. By \cite[2.5.2 Abstract Positivstellensatz]{mm},
there exist $l \in \NN$ and $p(\bar{\lambda}),q(\bar{\lambda})\in T^{R'}_{\tilde{\S}}$ such that
$\bar{\lambda}p(\bar{\lambda})=\bar{\lambda}^{2l}+q(\bar{\lambda})$. By the Cayley-Hamilton Theorem,
the natural homomorphism from $R[\lambda]$ to $R[F] \subset \mnr$ factors through $R'$.
It follows that $\F p(\F)=\F^{2l}+q(\F)$ and $p(\F),q(\F) \in \T^{R[\F]}_{\tilde{\S}} \subseteq \T^n_{\tilde{\S}} \subseteq \T^n_\S$.
Take $\b=p(\F)$ and $\c=q(\F)$.
\end{proof}

\begin{rem}
We can also give a less abstract but slightly longer
proof of the direction $(1) \Rightarrow (2)$
which is based on the observation that every ordering on
$R[\lambda]$ which contains $\tilde{\S}\cup \{-\lambda\}$, 
also contains the element $\det(\F-\lambda \idn)-(-\lambda)^n \in R[\lambda]$.
\end{rem}

\begin{rem}
Theorem \ref{thm2} does not seem to imply the Artin's theorem for matrix polynomials 
because the element $\b$ in its assertion (2) may not be central. 
Recall that
Artin's theorem for matrix polynomials 
says that for every $\a \in \sn(\RR[x_1,\ldots,x_d])$ 
which is positive semidefinite in every $x \in \RR^d$ there exists a nonzero $c \in \RR[x_1,\ldots,x_d]$
such that $\a c^2 \in \Sigma_n(\RR[x_1,\ldots,x_d])$.
It was proved independently in \cite{gr} and \cite{ps}. 
For the first constructive proof see \cite{sch2}.
We can deduce it from the formula (\ref{eq3}) 
in the proof of Proposition \ref{prop:ks}. 
\end{rem}

The following example shows that there exist $\F$ and $\S$ 
such that no element $\b$ from the assertion (2) of Theorem \ref{thm2} is central.

\begin{ex}
\label{noncentral}
Take $\S=\{x^3\}$ and $$\F=\left[ \begin{array}{cc} x & 0 \\ 0 & 1 \end{array} \right].$$
Clearly, $\F$ is positive semidefinite on $\K_\S$. Howewer there is no (central!) $b \in T_S$ such that $\F b=\F^{2k}+\c$
for some $k \in \NN$ and $\c=[c_{ij}] \in \T_\S^n$. Namely, if such a $b$ exists, then $xb=x^{2k}+c_{11}$ and $b=1+c_{22}$ for some
$c_{11}=u_1+x^3 v_1$ and $c_{22}=u_2+x^3 v_2$ with $u_i,v_i \in \sum \RR[x]^2$, then 
$x(1+u_2(x)+x^3 v_2(x))=x^{2k}+u_2(x)+x^3 v_2(x)$. Since $x$ divides $x^{2k}+u_2(x)$ which belongs to $\sum \RR[x]^2$,
it follows that also $x^2$ divides $x^{2k}+u_2(x)$. After canceling $x$ on both sides, we get that
the right-hand side is divisible by $x$ while the left-hand side is not.
\end{ex}

If we take $\F=-\idn$ in Theorem \ref{thm2}, we obtain the following result 
(which can also be proved more directly):

\begin{cor}
\label{cor1}
For every finite subset $\S$ of $\snr$, the following are equivalent:
\begin{enumerate}
\item[(1)] There is no homomorphism $\phi$ from $R$ into a real closed field
such that $\phi_n(\s)$ is positive semidefinite for every $\s \in \S$.
\item[(1')] There is no ordering on $\mnr$ which contains $\S$.
\item[(2)] $-\idn \in \T_\S^n$.
\end{enumerate}
If $R$ is an affine $\RR$-algebra, then (1), (1') and (2) are equivalent to
\begin{enumerate}
\item[(1'')] There is no $\phi \in V_\RR(R)$ such that $\phi_n(\s)$ is positive semidefinite
for every $\s \in \S$.
\end{enumerate}
\end{cor}

If we replace the set $\S$ in Corollary \ref{cor1} with the set $\S \cup \{-\F\}$, we get the following:

\begin{cor}(Nichtnegativsemidefinitheitsstellensatz)
\label{cor2}
For every finite subset $\S$ of $\snr$ and every $\F \in \snr$ the following are equivalent:
\begin{enumerate}
\item[(1)] For every real closed field $\kappa$ and  every $\phi \in V_\kappa(R)$
such that $\phi_n(\s)$ is positive semidefinite for every $\s \in \S$, we have
that $\phi_n(\F)$ is not negative semidefinite.
\item[(1')] Every ordering on $\mnr$ which contains $\S$ does not contain $-\F$.
\item[(2)] $-\idn \in \T_{\S \cup \{-\F\}}^n$.
\end{enumerate}
If $R$ is an affine $\RR$-algebra, then (1), (1') and (2) are equivalent to
\begin{enumerate}
\item[(1'')] For every $\phi \in V_\RR(R)$
such that $\phi_n(\s)$ is positive semidefinite for every $\s \in \S$, $\phi_n(\F)$ is not negative semidefinite.
\end{enumerate}
\end{cor}

By a slight modification of the proof of Theorem \ref{thm2} 
(replace the Krivine-Stengle Nichtnegativstellensatz for $R'$ with the Nullstellensatz)
we get the following generalization of 
Krivine-Stengle Nullstellensatz.

\begin{thm}(Semialgebraic Nullstellensatz)
\label{thm3}
For every finite subset $\S$ of $\snr$ and every $\F \in \snr$ the following are equivalent:
\begin{enumerate}
\item[(1)]
For every real closed field $\kappa$ and every $\phi \in V_\kappa(R)$
such that $\phi_n(\s)$ is positive semidefinite for every $\s \in \S$, we have
that $\phi_n(\F)=0$. 
\item[(1')]
For every ordering $Q \in \sper \mnr$ which contains $\S$, we have that $\F \in \supp \Q$.
\item[(2)] There exists $k \in \NN$ such that $-\F^{2k} \in \T_\S^n$.
\end{enumerate}
If $R$ is an affine $\RR$-algebra, then (1), (1') and (2) are equivalent to
\begin{enumerate}
\item[(1'')] For every $\phi \in V_\RR(R)$
such that $\phi_n(\s)$ is positive semidefinite for every $\s \in \S$, we have that $\phi_n(\F)=0$.
\end{enumerate}
\end{thm}

For $n=1$, semialgebraic Nullstellensatz follows from the Nichtnegativstellensatz by replacing
$\F$ with $-\F^T \F$. This does not work here even if the element $\b$ from the
assertion (2) of Theorem \ref{thm2} is central, see Example \ref{ex2}.

\begin{cor}(Real Nullstellensatz)
For every finite subset $\S$ of $\mnr$ and every $\F \in \mnr$ the following are equivalent:
\begin{enumerate}
\item[(1)] 
For every real closed field $\kappa$ and for every $\phi \in V_\kappa(R)$
such that $\phi_n(\s)=0$ for every $\s \in \S$, we have
that $\phi_n(\F)=0$. 
\item[(1')] For every ordering $\Q \in \sper \mnr$ such that $\s^T \s \in \supp \Q$ 
for every $\s \in \S$, we have that $\F^T \F \in \supp \Q$.
\item[(2)] There exists $l \in \NN$ such that $-(\F^T \F)^{l} \in \Sigma_n(R)+\operatorname{ideal}(\S)$
where $\operatorname{ideal}(\S)$ is the two-sided ideal in $\mnr$ generated by  $\S$.
\end{enumerate}
If $R$ is an affine $\RR$-algebra, then (1), (1') and (2) are equivalent to
\begin{enumerate}
\item[(1'')] For every $\phi \in V_\RR(R)$
such that $\phi_n(\s)=0$ for every $\s \in \S$, we have that $\phi_n(\F)=0$.
\end{enumerate}
\end{cor}

\begin{proof}
Clearly, (1) is equivalent to the following:
\begin{enumerate}
\item[(*)] For every real closed field $\kappa$ and every 
$\phi \in V_\kappa(R)$ such that $\phi_n(\s^T \s)=0$ 
for every $\s \in \S$, we have that $\phi_n(\F^T \F)=0$.
\end{enumerate}
By the comment after Lemma \ref{indord}, (*) is equivalent to (1').
By Theorem \ref{artinlang}, (1') is equivalent to a special case
of (*) for $\kappa=\RR$. The latter is clearly equivalent to (1'').

For the implication (1) $\Rightarrow$ (2), we use Theorem \ref{thm3} with
$\S' =\{-\s^T \s \mid \s \in \S\}$ instead of $\S$ and $\F^T \F$ instead of $\F$. 
We get that for some positive integer $k$, $-(\F^T \F)^{2k} \in \T_{\S'}^n$.
Clearly, $\T_{\S'}^n \subseteq \Sigma_n(R)+\operatorname{ideal}(\S') 
\subseteq \Sigma_n(R)+\operatorname{ideal}(\S)$, hence 
$-(\F^T \F)^{2k} \in \Sigma_n(R)+\operatorname{ideal}(\S)$.

Conversely, suppose that (2) is true and pick a homomorphism
$\phi$ from $R$ into a real closed field $\kappa$ such that $\phi_n(\s)=0$
for every $\s \in \S$. It follows that $\phi_n(\Sigma_n(R)+\operatorname{ideal}(\S))
\subseteq \sn(\kappa)^2$. In particular,
$\phi_n(-(\F^\ast \F)^{2k}) \in \sn(\kappa)^2$, and so $\phi_n(\F^\ast \F)=0$,
proving (1).
\end{proof}

\end{document}